\documentclass[12pt]{article}
\usepackage{graphicx}
\usepackage{graphicx, amssymb}
\usepackage[centertags]{amsmath}
\usepackage{amsfonts}
\usepackage{amssymb}
\usepackage{amsthm}
\usepackage{newlfont}
\begin{document}
\makeatletter

\def\endofproofmark{$\Box$}

\def\RR{I\!\!R}
\def\NN{I\!\!N}

\def\endofproofmark{$\Box$}

\renewcommand{\thefootnote}{\fnsymbol{footnote}}

\title{Paired domination and 2- distance Paired domination of the flower graph $f_{n\times
m}$}
\author{
Tanveer Iqbal, Syed Ahtsham Ul Haq Bokhary\\
Centre for Advanced Studies in Pure and Applied Mathematics,\\
Bahauddin Zakariya University, Multan, Pakistan\\
E-mail: {\tt tanveeriqbal203@gmail.com, sihtsham@gmail.com}\\
}
\date{} \maketitle
\pagestyle{myheadings}
\newtheorem{theo}{Theorem}[section]
\newtheorem{prop}[theo]{Proposition}
\newtheorem{lemma}[theo]{Lemma}
\newtheorem{defn}{Definition}[section]
\newtheorem{cor}[theo]{Corollary}
\newtheorem{problem}{Problem}
\def\frameqed{\framebox(5.2,6.2){}}
\def\deshqed{\dashbox{2.71}(3.5,9.00){}}
\def\ruleqed{\rule{5.25\unitlength}{9.75\unitlength}}
\def\myqed{\rule{8.00\unitlength}{12.00\unitlength}}
\def\qed{\hbox{\hskip 6pt\vrule width 7pt height11pt depth1pt\hskip 3pt}
\bigskip}
\thispagestyle{empty} \null \addtolength{\textheight}{1cm}
\begin{abstract}
Let $G = (V, E)$ be a graph without an isolated vertex. A set
$D\subseteq V(G)$ is a $k$-distance paired domination set of $G$ if
$D$ is a $k$-distance dominating set of $G$ and the induced subgraph
$\langle D \rangle$ has a perfect matching. The minimum cardinality
of a $k$-distance paired dominating set for graph $G$ is the
$k$-distance paired domination number, denoted by $\gamma_{p}
^{k}(G)$. In this paper, the $k$-distance paired domination of the
flower graph $f_{n\times m}$ is discussed. For $m,n\geq 3$, the
exact values for paired domination number and $2$-distance paired
domination number of flower graph $f_{n\times m}$ are determined .
\end{abstract}
Keywords:{\em domination number, paired domination number, flower graph}\\
Mathematics Subject Classification: 05C15, 05C65
\section{Introduction}
All the graphs considerd  in this paper are finite and simple. Let $G = (V, E)$ be a graph without an isolated vertex.
  A set $D\subseteq V(G)$ is said to be a dominating set if every vertex in $V(G)-D$
is adjacent to at least one vertex in $D$. A paired dominating is a paired dominating set of G if it is dominating and the induced subgraph $\langle D
\rangle$ has a perfect matching. This type of domination was introduced by Haynes and Slater in $\cite{TP, WJ}$ and is well studied, for example $\cite{BM,
PSM, HLMD}$.
\par For two vertices $x$ and $y$, let $d(x,y)$ denote the distance
between $x$ and $y$ in $G$. A set $D \subseteq V(G)$ is a \emph{k}-distance dominating set of $G$ if every vertex in $V(G)-D$ is within distance $\emph{k}$
of at least one vertex in $D$. The \emph{k}-distance domination number $\gamma^{k}(G)$ of \emph{G} is the minimum cardinality among all \emph{k}-distance
dominating sets of $G$. The \emph{k}-distance paired-domination was introduced by Joanna Raczek $\cite{JR}$ as a generalization of paired-domination. For a
positive integer $k$, a set $D \subseteq V(G)$ is a \emph{k}-distance paired-dominating set if every vertex in $V(G)-D$ is within distance $k$ of a vertex
in $D$ and the induced subgraph $\langle D\rangle$ has a perfect matching. The \emph{k}-distance paired-domination number, denoted by $\gamma_{p}^{k}(G)$
is the minimum cardinality of a k-distance paired-dominating set. The distance paired domination number of different families of graph such as generalized
Peterson graphs, circulant graphs were studies in $\cite{HXYK, HXYGK}$.
\par In this paper, paired domination number and $2$-distance paired domination number of $f_{n\times m}$ are studied.
The exact values of the paired dominating number has been found for every value of $m$ and $n$.
 Throughout the paper, the subscripts are taken
modulo $n$ when it is unambiguous.
\section{Paired domination number of flower graph $f_{n\times m}$}
A graph \emph{G} is called an $f_{n \times m}$ flower graph if it
has $n$ vertices which form an $n$-cycle and $n$ sets of $m-2$
vertices which from $m$-cycles around the $n$-cycle, so that each
$m$-cycle uniquely intersects the $n$-cycle on a single edge. Let
$C_{1, m}, C_{2, m}, C_{3, m},\dots, C_{n, m}$ are edge disjoint
outer cycles of length $m$. Every two consecutive outer cycles has a
common vertex of degree four.
In each cycle, there are $m-2$ vertices of degree two and two vertices of degree four.\\
 This graph will be denoted by $f_{n \times m}$. It is clear that $f_{n\times m}$ has $n(m-1)$ vertices and $nm$ edges. The $m$-cycles are
called the petals and the $n$-cycles is called center of $f_{n\times m}$. The $n$ vertices which form the center are all of degree $4$ and all other
vertices have degree $2$. The centered vertices are denoted by $u_i$, where $i = 1,\dots,n$. The vertices of outer cycles are denoted by $v_{ij}$, where $
1\leq i \leq n$ and $ 1\leq j \leq m-2$. Thus, the vertex and edge set of the flower graph $f_{n\times m}$ is
\begin{center} $ V(f_{n\times m}) =\{u_{i}, v_{i,j} : 1\leq i \leq
n, 1\leq j \leq m-2\}$
\end{center}
$$ E(f_{n\times m}) = E_{1}\cup E_{2}\cup E_{3},$$ where
$ E_{1} = \{u_{i}u_{i+1}: 1\leq i \leq n\}$, $ E_{2} = \{v_{i,j}v_{i,j+1}: 1\leq i \leq n, 1\leq j \leq m-3\}$ and $ E_{3} = \{u_{i}v_{i, 1}, u_{i+1}v_{i,
m-2} : 1\leq i \leq n\}$.
\par  Let $D_{p} = \{x_{i}, y_{i} : i=1,2,...,q\}$ be an
arbitrary paired dominating set of the flower graph $f_{n\times m}$.
For convenience, let $V_i$ is the set of vertices of the outer
cycles $C_{i, m}$, for each $i=1,\dots,n$ and $U$ is the set of
vertices of the inner cycle. Thus,
\begin{center}
$V_{i}=\{v_{i,j}$ $\in$ $V(f_{n\times m})$ $:$ $deg(v_{i,j})= 2 : 1 \leq i \leq n, 1 \leq j \leq m-2\}$ \end{center}
\begin{center} $U=\{u_{i}$ $\in$ $V(f_{n\times m})$ $:$ $deg(u_{i})= 4 : 1 \leq i
\leq n\}$ \end{center} and let
\begin{center}
$D_{vv}= \{(x_{i}, y_{i})$ $\in$ $D_{p}$ : $x_{i}$ $\in$ $V_{i}$, $y_{i}$ $\in$ $V_{i}\}$,
\end{center}
\begin{center}
$D_{uu}= \{(x_{i}, y_{i})$ $\in$ $D_{p}$ : $x_{i}$ $\in$ $U$, $y_{i}$ $\in$ $U\}$,
\end{center}
\begin{center}
$D_{vu}= \{(x_{i}, y_{i})$ $\in$ $D_{p}$ : $x_{i}$ $\in$ $V_{i}$, $y_{i}$ $\in$ $U\}$.
\end{center}
Obviously, $D_{p}= D_{vv} \cup D_{uu} \cup D_{vu}$.\\
\begin{lemma}\label{lb of Dp} Let $D_{p}$ be a paired dominating set of the graph $f_{n \times m}$ and $V_{i}$
be the set of vertices of degree $2$ of the outer m-cycles $C_{i,
m}$. Then $D_{p}$ contain at least $2\lceil
\frac{m-2(k+1)}{2(k+1)}\rceil$ vertices from each $V_{i}$.
\end{lemma}
\begin{proof}
Since, each $C_{i, m}$ has $2$ vertices of degree $4$ and these
vertices can dominate at most $2k$ $(\forall\,\ k\geq 1)$ vertices
of each $V_{i}$. Therefore, to dominate the remaining $m-2(k+1)$
vertices of each $V_{i}$, we need at least $2\lceil
\frac{m-2(k+1)}{2(k+1)}\rceil$ vertices in $D_{p}$ from each set
$V_{i}$.
\end{proof}
In the next theorem, the exact value of the paired dominating number
of the graph $f_{n \times m}$ for $m \equiv 0,1, 2, 3$
$(\textrm{mod}\ 4)$ are given.\\
\begin{theo} For $m, n \geq 3$,\\
$$\gamma_{p}(f_{n \times m})= \left\{
                         \begin{array}{ll}
2\lceil \frac{nm-2n}{4}\rceil,  & if\,\,\,  m \equiv
0\,\ (\textrm{mod}\ 4)\\\\
2\lceil \frac{nm-n}{4}\rceil,   &    if\,\,\,  m \equiv 1,
2\,\ (\textrm{mod}\ 4)\\\\
2\lceil \frac{3nm-5n}{12}\rceil,    &   if\,\,\,  m \equiv
3\,\ (\textrm{mod}\ 4)\\
\end{array}
                        \right.$$
\end{theo}
\begin{proof}
Let $t = \lfloor
\frac{m}{4}\rfloor$. We prove this theorem by giving the following cases.\\
\textbf{Case 1:}  $m \equiv 0\,\ (\textrm{mod}\ 4)$.\\
If $n$ is even,  define $t^{'}=\frac{n}{2}$.
The set $D_p$ for $m=4$ is defined as follows:\\
$D_{p}=\{u_{2l-1}, u_{2l}\  :\  1\leq l\leq
  t^{'}\}.$\\
If $m \equiv 0 \,\ (\textrm{mod}\ 4)$  (where $m \neq 4$), define\\
  $D_p=\{v_{i, 4j-1}, v_{i, 4j} :\ 1\leq i\leq n,\,\, 1\leq j\leq t-1\}\cup
  \{u_{2l-1}, u_{2l}\  :\  1\leq l\leq
  t^{'}\}.$\\
If $n$ is odd, then define $t^{'}=\frac{n-1}{2}$. The set $D_p$ for $m=4$ is defined as follows:\\
$D_{p}=\{u_{2l-1}, u_{2l}\  :\  1\leq l\leq  t^{'}\} \cup\{u_{n}, v_{n,1}\} $\\
If $m \equiv 0 \,\ (\textrm{mod}\ 4)$ (where $m \neq 4$), define\\
  $D_p=\{v_{i, 4j-1}, v_{i, 4j}  :\ 1\leq i\leq n-1,\,\, 1\leq j\leq t-1\}\cup\{v_{n, 4j}, v_{n,
  4j+1}\}\cup \{u_{2l-1}, u_{2l}\: \,\,\,\ 1\leq l\leq
  t^{'}\}\cup\{u_{n}, v_{n, 1}\}.$\\
In each case, it is easy to verify that $D_p$ is a paired dominating set. The cardinality of $D_{p}$ in each case is $2\lceil \frac{nm-2n}{4}\rceil$. Hence,\\
\begin{equation}
\begin{split}
 \gamma_{p}(f_{n\times m})\leq 2\lceil \frac{nm-2n}{4}\rceil.
 \end{split}
 \end{equation}
 To prove the lower bound for the paired dominating set $D_{p}$. Let $D_p =\{x_{i}, y_{i}:1\leq i\leq q\}$ be a paired dominating set of $f_{n\times m}$.
 By Lemma \ref{lb of Dp}, $D_{p}$ contain at least $2\lceil
\frac{m-4}{4}\rceil$ pair of vertices from each $V_{i}$. With loss
of generality, we can suppose that $v_{i, 1}$ and $v_{i, m-2}$ are
the vertices which are yet to be dominated in $V_{i}$. Then, to
dominate $v_{i, 1}$ and $v_{i}$ either $v_{i, 1}, v_{i, m-2}\in
D_{p}$ or $u_{i}, u_{i+1}\in D_{p}$. In both these cases, each
$C_{i, m}$ has at least two vertices belong to $D_{p}$. Since, each
vertex of degree $4$ belong to neighboring cycle, therefore each
vertex of degree $4$ belong to $D_{p}$. Further, $\langle
D_{p}\rangle$ is a perfect matching. Thus only edges that are not
adjacent to each other can belong to $D_{p}$. There are $\lceil
\frac{n}{2}\rceil$ non adjacent edges of type$(4, 4)$ if $n$ is even
and $\lceil
\frac{n}{2}\rceil-1$ edges if $n$ is odd. In the later case, one edge of the type $(2, 4)$ also belong to $D_{p}$. Thus\\
\begin{equation*}
\begin{split}
 q & \geq n\lceil \frac{m-4}{4}\rceil + \frac{n}{2}\\
 &=\frac{nm-2n}{4} \\
\end{split}
 \end{equation*}
which implies that $ q \geq \lceil \frac{nm-2n}{4}\rceil$. Hence
\begin{equation}
\begin{split}
\gamma_{p}(f_{n\times m}) \geq 2\lceil \frac{nm-2n}{4}\rceil.
\end{split}
\end{equation}\\ From Equation $1$ and $2$, it is clear that $$\gamma_{p}(f_{n \times m})=
2\lceil \frac{nm-2n}{4}\rceil.$$
 \textbf{Case 2:} $m \equiv 1\,\ (\textrm{mod}\ 4)$.\\
In this case, define the set $D_p$ as follows:\\
$D_p=\{v_{i, 4j-3}, v_{i, 4j-2} :\ 1\leq i\leq n,\,\, 1\leq
  j\leq t\}.$\\
It is easy to see that $D_p$ is a paired dominating set and the cardinality of paired dominating set is $2\lceil \frac{nm-n}{4}\rceil$. Hence,\\
\begin{equation}
\begin{split}
 \gamma_{p}(f_{n\times m}) \leq 2\lceil \frac{nm-n}{4}\rceil.
 \end{split}
 \end{equation}
The lower bound  of paired dominating set is proved in the following way.\\
Let $D_{p}=\{x_{i}, y_{i} :1 \leq i\leq q\}$ be a paired dominating
set. By Lemma \ref{lb of Dp}, $D_{p}$ contain at least $\lceil
\frac{m-4}{4}\rceil$ vertices from each $V_{i}$ of $C_{i, m}$. If
$m\equiv1,$ $(\textrm{mod}\ 4)$, then $\lceil \frac{m-4}{4}\rceil$
pair of  vertices
dominate $m-1$ vertices from each $V_{i}$ of $C_{i, m}$, $where 1 \leq i\leq n.$ Therefore\\
\begin{equation*}
\begin{split}
q & \geq n\lceil \frac{m-4}{4}\rceil\\
 &=n(\frac{m-1}{4}) \\
 &=\frac{nm-n}{4}
\end{split}
 \end{equation*}
which implies that $q  \geq \lceil \frac{nm-n}{4}\rceil$. Hence\\
\begin{equation}
\begin{split}
 \gamma_{p}(f_{n\times m})\geq 2\lceil \frac{nm-n}{4}\rceil.
\end{split}
\end{equation}\\ Equation $3$ and $4$ implies that $$\gamma_{p}(f_{n \times m})=
2\lceil \frac{nm-2n}{4}\rceil.$$\\
 \textbf{Case 3:}  $m \equiv 2\,\ (\textrm{mod}\ 4).$\\
 Let $t^{'}=\lceil\frac{n}{4}\rceil$. If $ n =5$, define\\
$D_p=\{u_{1}, u_{2}, u_{4}, u_{5},v_{i, 4j-2}, v_{i, 4j-1}:\ 1\leq i\leq n,\ 1\leq j\leq t\}.$\\
If $n\neq 5$, define\\
  $D_p=\{v_{i, 4j-2}, v_{i, 4j-1}:\
1\leq i\leq n,\,\, 1\leq j\leq
 t\}\cup\{u_{4l-3}, u_{4l-2 }  : \ 1\leq l\leq
 t^{'}\}.$\\It is easy to see that $D_p$ is a paired dominating set and the cardinality of $D_p$ is $2\lceil \frac{nm-n}{4}\rceil$. Hence\\
\begin{equation}
\begin{split}
 \gamma_{p}(f_{n\times m}) \leq 2\lceil \frac{nm-n}{4}\rceil.
 \end{split}
 \end{equation}
To prove the lower bound, let $D_{p}=\{x_{i}, y_{i} :1 \leq i\leq
q\}$ be a paired dominating set. By Lemma \ref{lb of Dp}, $D_{p}$
contain at least $\lceil \frac{m-4}{4}\rceil$ vertices from each
$V_{i}$ of $C_{i, m}$.If $m\equiv2,$ $(\textrm{mod}\ 4)$, then
$\lceil \frac{m-4}{4}\rceil$ pair of  vertices dominate $m-2$
vertices from each $V_{i}$ of $C_{i, m}$, $where 1 \leq i\leq n.$
The only vertices which are yet to be dominated are the vertices
$u_{i}$ of degree $4$. Since there are $n$ vertices of degree $4$,
therefore we need at least $\lceil \frac{n}{4}\rceil$ more pair of vertices in $D_{p}$. Thus\\
\begin{equation*}
\begin{split}
|D_{p}| & \geq n\lceil \frac{m-4}{4}\rceil+\lceil \frac{n}{4}\rceil\\
 &=n(\frac{m-2}{4})+\lceil \frac{n}{4}\rceil \\
 &=\frac{nm-2n}{4}+\lceil \frac{n}{4}\rceil\\
 &=\lceil \frac{nm-2n+n}{4}\rceil\\
 &=\lceil \frac{nm-n}{4}\rceil
\end{split}
 \end{equation*}
which implies that $q  \geq \lceil \frac{nm-n}{4}\rceil$. Hence\\
\begin{equation}
\begin{split}
\gamma_{p}(f_{n\times m}) \geq 2\lceil \frac{nm-n}{4}\rceil.
\end{split}
\end{equation}\\ Equation $5$ and $6$ implies that $$\gamma_{p}(f_{n \times m})=
2\lceil \frac{nm-2n}{4}\rceil.$$\\
\textbf{Case 4:} For $m \equiv 3 \,\ (\textrm{mod}\ 4)$.\\
For $n \equiv 0, 2 \,\ (\textrm{mod}\ 3)$, let $t^{'} = \lceil
\frac{n}{3}\rceil$ and for $n \equiv 1\,\ (\textrm{mod}\ 3)$,
$t^{'} = \lfloor \frac{n}{3}\rfloor$.\\ If $n = 4$, define\\
$D_{p}=\{v_{1, 4j-1}, v_{1, 4j}, v_{2, 4j-1}, v_{2, 4j}, v_{3, 4j-1}, v_{3, 4j}, v_{4, 4j-1}, v_{4, 4j} :\
 1\leq j\leq t\}\cup\{u_{1},u_{2}, u_{3},u_{4}\}$\\
If  $n = 3t, \forall$ $t \geq 1$, then define\\
$D_{p}=\{v_{3i-2, 4j-1}, v_{3i-2, 4j}, v_{3i-1, 4j-1}, v_{3i-1, 4j} :\
 1\leq i\leq t^{'},\,\, 1\leq j\leq t\}\cup\{v_{3l, 4j-2}, v_{3l,
 4j-1} :\ 1\leq i\leq t^{'}\}\cup\{u_{3l^{'}-2},u_{3l^{'}-1} :\
 1\leq l^{'}\leq t^{'}\}.$\\
If  $n = 3t+1, \forall$ $t \geq 2$, then define\\
$D_{p}=\{v_{3i-2, 4j-1}, v_{3i-2, 4j}, v_{3i-1, 4j-1}, v_{3i-1, 4j} :\
 1\leq i\leq t^{'},\,\, 1\leq j\leq t\}\cup\{v_{3l, 4j-2}, v_{3l,
 4j-1} :\ 1\leq i\leq t^{'}-1\}\cup\{v_{n-1, 4j-1}, v_{n-1,
 4j}, v_{n, 4j-1}, v_{n, 4j}\}\cup\{u_{3l^{'}-2},u_{3l^{'}-1} :\
 1\leq l^{'}\leq t^{'}\}\cup\{u_{n-1},u_{n}\}.$\\
If  $n = 3t+2, \forall$ $t \geq 1$, then define\\
$D_{p}=\{v_{3i-2, 4j-1}, v_{3i-2, 4j}, v_{3i-1, 4j-1}, v_{3i-1, 4j} :\
 1\leq i\leq t^{'},\,\, 1\leq j\leq t\}\cup\{v_{3l, 4j-2}, v_{3l,
 4j-1} :\ 1\leq i\leq t^{'}-1\}\cup\{u_{3l^{'}-2},u_{3l^{'}-1} :\
 1\leq l^{'}\leq t^{'}\}.$\\
It is easy to see that $D_p$ is a paired dominating set in each case and the cardinality of $D_p$ is $2\lceil \frac{3nm-5n}{12}\rceil$. Hence,\\
\begin{equation}
\begin{split}
 \gamma_{p}(f_{n\times m}) \leq 2\lceil \frac{3nm-5n}{12}\rceil.
 \end{split}
 \end{equation}
Now we prove the lower bound of paired dominating set.\\
Let $D_{p}=\{x_{i}, y_{i} :1 \leq i\leq q\}$ be a paired dominating
set. By Lemma \ref{lb of Dp}, $D_{p}$ contain at least $\lceil
\frac{m-4}{4}\rceil$ pair of vertices from each $V_{i}$ of $C_{i,
m}$. The graph $f_{n \times m}$ has vertices of degree $2$ and $4$.
The $D_{p}$ can contain the edges of types $(2, 2), (2, 4) and(4,
4)$. The edge of the type $(2, 2), (2, 4)$ and $(4, 4)$ can dominate
$4, 6$ and $8$ vertices respectively of $f_{n \times m}$. Each
$C_{i}$ contain $m-2$ vertices of degree $2$ and $2$ vertices of
degree $4$.
 Since any pair of $D_{p}$ can dominate
at most $4$ vertices of each $C_{i}$. Therefore, to dominate remaining $m-4$ vertices, we need at least $\lceil \frac{m-4}{4}$ pairs of adjacent vertices
in the $D_{p}$. Since $m\equiv 3$ $(\textrm{mod}\ 4)$, therefore these $\lceil \frac{m-4}{4}\rceil$ pairs of adjacent vertices dominate $m-3$ vertices in
each $C_{i}$. Also each edge of the type $(4, 4)$ dominate $8$ vertices. So we have to choose at least one edge from $3$ consecutive copies of outer $C_{i,
m}$. This implies that
\begin{equation*}
\begin{split}
|D_{p}|& \geq n\lceil \frac{m-4}{4}\rceil + \lceil \frac{n}{3}\rceil\\
&= n(\frac{m-3}{4}) + \lceil \frac{n}{3}\rceil\\
& =  \lceil \frac{n}{3} + \frac{nm-3n}{4}\rceil\\
&= \lceil \frac{3nm-5n}{12}\rceil
\end{split}
\end{equation*}
which implies that $q  \geq \lceil \frac{3nm-5n}{12}\rceil$. Hence\\
\begin{equation}
\begin{split}
\gamma_{p}(f_{n\times m}) \geq 2\lceil \frac{3nm-5n}{12}\rceil.
\end{split}
\end{equation}\\ From Equation $7$ and $8$, it is clear that $$\gamma_{p}(f_{n \times m})=
2\lceil \frac{3nm-5n}{12}\rceil.$$ In Figure $1$, we show the paired
dominating set of $f_{n\times m}$ for different values of $n$ and
$m$, where the vertices of paired dominating set are in dark.
\begin{figure}[!ht]
\begin{center}
        \centerline
         {\includegraphics[width=12cm]{f4,9.md}}
        \caption{The paired dominating set of $f_{n\times m}$ }
        \end{center}
            \end{figure}
\end{proof}
\section{2-distance paired domination number of flower graph $f_{n\times m}$ }
 In this section, the exact value of $2$-distance paired domination number of flower graph $f_{n\times m}$ is determined.
\begin{theo}
For $m, n \geq 3$,
$$\gamma_{p}^{2}(f_{n \times m})= \left\{
                         \begin{array}{ll}
2\lceil \frac{nm-3n}{6}\rceil    ,\,\,\,      if\,\,\,  m \equiv
0, 5\,\ (\textrm{mod}\ 6)\\\\
2\lceil \frac{nm-n}{6}\rceil    ,\,\,\,      if\,\,\,  m \equiv 1,
2\,\ (\textrm{mod}\ 6)\\\\
2\lceil \frac{5nm-9n}{30}\rceil    ,\,\,\,      if\,\,\,  m \equiv
3\,\ (\textrm{mod}\  6)\\\\
2\lceil \frac{2nm-5n}{12}\rceil    ,\,\,\,      if\,\,\,  m \equiv
4\,\ (\textrm{mod}\ 6)\\
\end{array}
                        \right.$$
\end{theo}
\begin{proof}
Let  $t = \lfloor \frac{m}{6}\rfloor$. We have the following cases.\\
 \textbf{Case 1:} For $m \equiv 0 \,\ (\textrm{mod}\ 6)$.\\
Define $t^{'}= \lceil\frac{n}{2}\rceil$. The $2$-paired dominating set for $ m =6$ and $n=2t+1$, $\forall$ $t \geq 1$  is defined as:\\
$D_{2,p}=\{u_{2i-1}, u_{2i}  :\ 1\leq i\leq t'-1 \}\cup\{u_{n}, v_{n, 1}\}.$\\
For $ m =6$ and $n=2t$, $\forall$ $t \geq 2$, define\\
$D_{2,p}=\{u_{2i-1}, u_{2i}  :\ 1\leq i\leq t' \}.$\\
If $m\neq 6$ and $n=2t+1$, $\forall$ $t \geq 1$, then define\\
$D_{2,p}=\{v_{i, 6j-1}, v_{i, 6j} :\ 1\leq j\leq t-1,\,\ 1\leq i\leq
n-1\} \cup \{v_{n, 6j}, v_{n, 6j+1} : 1\leq j \leq
t-1\} \cup \{u_{n}, v_{n, 1}\} \cup \{u_{2l-1}, u_{2l}\ :\  1\leq l\leq t'-1\}.$\\
If $ m =6$ and $n=2t$, $\forall$ $t \geq 2$, then define\\
$D_{2,p}=\{v_{i, 6j-1}, v_{i, 6j} :\ 1\leq j\leq t-1,\,\ 1\leq i\leq
n\}\cup
  \{u_{2l-1}, u_{2l}\ :\  1\leq l\leq t'\}.$\\
In all these possibilities, it is easy to verify that $D_{2,p}$ is a
$2$-paired dominating set. Further, the cardinality of $D_{2,p}$ in
each case is $2\lceil \frac{nm-3n}{6}\rceil$. Hence,
\begin{equation}
\begin{split}
 \gamma_{p}^{2}(f_{n \times m}) \leq 2\lceil \frac{nm-3n}{6}\rceil.
 \end{split}
 \end{equation}
Now to prove the lower bound for $2$-distance paired dominating set.
Let $D_{2,p} =\{x_{i}, y_{i}:1\leq i\leq q\}$ be a paired dominating
set of $f_{n\times m}$.  By Lemma \ref{lb of Dp}, $D_{2,p}$ contains
at least $\lceil\frac{m-6}{6}\rceil$ vertices from each $V_{i}$
$(m\neq 6)$.
 These $\lceil\frac{m-6}{6}\rceil$ vertices dominate $m-6$ vertices of degree $2$ in each $C_{i, m}$.
 Suppose that $v_{i, 1}, v_{i, 2}, v_{i, m-2}$  and $v_{i, m-3}$ are the vertices which are yet to be dominated.
 To dominate these vertices either vertices of degree $2$ belongs to $D_{2,p}$ or vertices of degree $4$ belongs to $D_{2,p}$.
 In both the cases each $C_{i}$ has at least two vertices which belongs to $D_{2,p}$. Since each vertex of degree $4$ belong to neighboring cycle, so it must
 belong to $D_{2,p}$. Also $\langle D_{2,p}\rangle$ has perfect matching which implies that the only  non adjacent edges of type $(4, 4)$ belong to $D_{2,p}$.
 There are $\lceil\frac{n}{2}\rceil$ non adjacent edges of type $(4,4)$ if $n$ is even and $\lceil\frac{n}{2}\rceil-1$ if $n$ is odd. In later case one
 edge of type $(2, 4)$ also belong to $D_{2,p}$. Thus
\begin{equation*}
\begin{split}
q & \geq n\lceil \frac{m-6}{6}\rceil + \lceil\frac{n}{2}\rceil\\
 &=\frac{nm-6n}{6}+ \lceil\frac{n}{2}\rceil \\
 &=\lceil\frac{nm-6n+3n}{6}\rceil\\
 &=\lceil\frac{nm-3n}{6}\rceil
\end{split}
 \end{equation*}
which implies that $q  \geq \lceil \frac{nm-3n}{6}\rceil$. Thus\\
\begin{equation}
\begin{split}
\gamma_{p}^{2}(f_{n \times m}) \geq 2\lceil \frac{nm-3n}{6}\rceil.
\end{split}
\end{equation}\\ From Equation $9$ and $10$, it is clear that $$\gamma_{p}^{2}(f_{n \times m})=
2\lceil \frac{nm-3n}{6}\rceil.$$\\
\textbf{Case 2:} For $m \equiv 1\,\ (\textrm{mod}\ 6)$.\\ Define\\
$D_{2,p}=\{v_{i, 6j-4}, v_{i, 6j-3} : 1\leq i\leq n,\,\ 1 \leq j \leq t\}.$\\
It is easy to verify that $D_{2,p}$ is a $2$-paired dominating set. Further, the cardinality of $D_{2,p}$ is
\begin{equation}
\begin{split}
\gamma_{p}^{2}(f_{n \times m}) \leq 2\lceil \frac{nm-n}{6}\rceil.
\end{split}
\end{equation}\\
Now to prove the lower bound of $2$-distance paired dominating set.
Let $D_{2,p}=\{x_{i}, y_{i} :1 \leq i\leq q\}$ be a paired
dominating set. By Lemma \ref{lb of Dp}, $D_{2,p}$ contain at least
$\lceil \frac{m-6}{6}\rceil$ vertices from each $V_{i}$ of $C_{i,
m}$. If $m\equiv1,$ $(\textrm{mod}\ 6)$, then $\lceil
\frac{m-6}{6}\rceil$ pair of  vertices
dominate $m-1$ vertices from each $V_{i}$ of $C_{i, m}$, where $1 \leq i\leq n.$ Therefore\\
\begin{equation*}
\begin{split}
q  & \geq n\lceil \frac{m-6}{6}\rceil\\
 &=n(\frac{m-1}{6}) \\
 &=\frac{nm-n}{6}\\
\end{split}
 \end{equation*}
which implies that $q  \geq \lceil \frac{nm-n}{6}\rceil$. Thus\\
\begin{equation}
\begin{split}
\gamma_{p}^{2}(f_{n \times m}) \geq 2\lceil \frac{nm-n}{6}\rceil.
\end{split}
\end{equation}\\ From Equation $11$ and $12$, it is clear that $$\gamma_{p}^{2}(f_{n \times m})=
2\lceil \frac{nm-n}{6}\rceil.$$\\
\textbf{Case 3:} $m \equiv 2\,\ (\textrm{mod}\ 6)$\\
Let $t^{'}= \lceil\frac{n}{6}\rceil.$ For $n \equiv 0, 2, 3, 4, 5 \,\ (\textrm{mod}\ 6)$, define\\
$D_{2,p}=\{v_{i, 6j-3}, v_{i, 6j-2} : 1\leq i\leq n,\,\ 1 \leq j \leq t\} \cup \{u_{6l-5}, u_{6l-4} : 1\leq l \leq t'\}.$\\
For $n \equiv 1 \,\ (\textrm{mod}\ 6)$, define\\
$D_{2,p}=\{v_{i, 6j-3}, v_{i, 6j-2} : 1\leq i\leq n,\,\ 1 \leq j \leq t\} \cup \{u_{6l-5}, u_{6l-4}, u_{n-1}, u_{n} : 1\leq l \leq t'\}.$\\
In all these possibilities, it is easy to verify that $D_{2,p}$ is a $2$-paired dominating set. Further, the cardinality of $D_{2,p}$ in each case is
 $2\lceil \frac{nm-n}{6}\rceil$. Hence,
\begin{equation}
\begin{split}
\gamma_{p}^{2}(f_{n \times m}) \leq 2\lceil \frac{nm-n}{6}\rceil.
\end{split}
\end{equation}\\
Now we give the lower bound of $2$-distance paired dominating set.
Let $D_{2,p}=\{x_{i}, y_{i} :1 \leq i\leq q\}$ be a paired
dominating set of $f_{n\times m}$. If $m\equiv2,$ $(\textrm{mod}\
6)$, then $\lceil \frac{m-6}{6}\rceil$ pair of vertices dominate
$m-2$ vertices from each $V_{i}$ of $C_{i, m}$, $where 1 \leq i\leq
n.$ The only vertices which are yet to be dominated are the vertices
$u_{i}$ of degree $4$. Since there are $n$ vertices of degree $4$,
therefore we need at least $\lceil \frac{n}{6}\rceil$ more pair of vertices in $D_{2,p}$. Thus\\
\begin{equation*}
\begin{split}
q & \geq n\lceil \frac{m-6}{6}\rceil+\lceil \frac{n}{6}\rceil\\
 &=n(\frac{m-2}{6})+\lceil \frac{n}{6}\rceil \\
 &=\frac{nm-2n}{6}+\lceil \frac{n}{6}\rceil\\
 &=\lceil \frac{nm-2n+n}{6}\rceil\\
 &=\lceil \frac{nm-n}{6}\rceil
\end{split}
 \end{equation*}
which implies that $q  \geq \lceil \frac{nm-n}{6}\rceil$. Therefore\\
\begin{equation}
\begin{split}
\gamma_{p}^{2}(f_{n \times m}) \geq 2\lceil \frac{nm-n}{6}\rceil.
\end{split}
\end{equation}\\
From Equation $13$ and $14$, it is clear that $$\gamma_{p}^{2}(f_{n \times m})=
2\lceil \frac{nm-n}{6}\rceil.$$\\
\textbf{Case 4}: $m\equiv 3 (\textrm{mod}\ 6)$.\\
let $t^{'}= \lceil\frac{n}{5}\rceil.$ For $n = 3$, define\\
$D_{2, p}=\{v_{1, 6j-1}, v_{1, 6j}, v_{2, 6j-1}, v_{2, 6j}, v_{3,
6j-2}, v_{3, 6j-1} : 1\leq i\leq t',\,\ 1 \leq j\leq t\}
  \cup \{u_{1}, u_{2}\}.$\\
For $n = 5$, define\\
$D_{2, p}=\{v_{1, 6j-1}, v_{1, 6j}, v_{2, 6j-1}, v_{2, 6j}, v_{3, 6j-2}, v_{3, 6j-1}, v_{4, 6j-3}, v_{4, 6j-2},
 v_{5, 6j-2}, v_{5, 6j-1} :1 \leq j\leq
t\}  \cup \{u_{1}, u_{2}\}.$\\
For $n = 4, 6$, define\\
$D_{2,p}=\{v_{4i-3, 6j-1}, v_{4i-3, 6j}, v_{4i-2, 6j-1}, v_{4i-2,
6j}, v_{3, 6j-2}, v_{3, 6j-1}, v_{4, 6j-2}, v_{4, 6j-1} : 1\leq
i\leq t',\,\ 1 \leq
j\leq t\}  \cup \{u_{1}, u_{2}\}.$\\
For $n = 5t, \forall$ $t \geq 2$, define\\
$D_{2,p}=\{v_{5i-4, 6j-1}, v_{5i-4, 6j}, v_{5i-3, 6j-1}, v_{5i-3, 6j}, v_{5i-2, 6j-2}, v_{5i-2, 6j-1}, v_{5i, 6j-2},\\ v_{5i, 6j-1}, v_{5i-1, 6j-3},
v_{5i-1, 6j-2} : 1\leq i\leq t',\,\ 1 \leq j\leq t\}  \cup \{u_{5l-4}, u_{5i-3} : 1 \leq l \leq t' \}.$\\
For $n =5t+1, \forall$ $t \geq 1$, define\\
$D_{2, p}=\{v_{5i-4, 6j-1}, v_{5i-4, 6j}, v_{5i-3, 6j-1}, v_{5i-3,
6j}, v_{5i-2, 6j-2}, v_{5i-2, 6j-1}, v_{5p, 6j-2},\\ v_{5p, 6j-1},
v_{5p-1, 6j-3}, v_{5p-1, 6j-2}, v_{n-2, 6j-2}, v_{n-2, 6j-1},
v_{n-1, 6j-1}, v_{n-1, 6j}, v_{n, 6j-1}, v_{n, 6j}: 1\leq i\leq
t'-1,\,\ 1 \leq j\leq t,\,\ 1\leq
p\leq t'-2\}  \cup \{u_{5l-4}, u_{5i-3}, u_{n-1}, u_{n} : 1 \leq l \leq t'-1 \}.$\\
For $n =5t+2\,\ \forall$ $t \geq 1$, define\\
$D_{2, p}=\{v_{5i-4, 6j-1}, v_{5i-4, 6j}, v_{5i-3, 6j-1}, v_{5i-3,
6j}: 1\leq i\leq t',\,\ 1 \leq j\leq t\} \cup\{v_{5i-2, 6j-2},
v_{5i-2, 6j-1},
v_{5i-1, 6j-3}, v_{5i-1, 6j-2}, v_{5i, 6j-2}, v_{5i, 6j-1} : 1\leq i \leq t'-1,\,\ 1\leq j\leq t \} \cup \{u_{5l-4}, u_{5i-3} : 1 \leq l \leq t' \}.$\\
For $n =5t+3\,\ \forall$ $t \geq 1$, define\\
$D_{2, p}=\{v_{5i-4, 6j-1}, v_{5i-4, 6j}, v_{5i-3, 6j-1}, v_{5i-3,
6j}, v_{5i'-1, 6j-3}, v_{5i'-1, 6j-2},: 1\leq i\leq t',\,\ 1\leq
i'\leq t'-1,\,\,1 \leq j\leq t\}\cup\{v_{5i-2, 6j-2}, v_{5i-2,
6j-1},\\ v_{5p, 6j-2}, v_{5p, 6j-1} :
 1\leq i \leq t',\,\ 1\leq j\leq t,\,\ 1\leq p \leq t'-1 \} \cup \{u_{5l-4}, u_{5i-3} : 1 \leq l \leq t' \}.$\\
For $n =5t+4\,\ \forall$ $t \geq 1$, define\\
$D_{2, p}=\{v_{5i-4, 6j-1}, v_{5i-4, 6j}, v_{5i-3, 6j-1}, v_{5i-3,
6j}, v_{5i'-1, 6j-3}, v_{5i'-1, 6j-2} : 1\leq i\leq t',\,\, 1\leq
i'\leq t'-1,\,\ 1 \leq j\leq t\} \cup \{v_{5i-2, 6j-2}, v_{5i-2,
6j-1}, v_{5p, 6j-2}, v_{5p, 6j-1}, v_{n, 6j-2}, v_{n, 6j-1}
:1\leq i \leq t',\,\ 1\leq j\leq t,\,\ 1\leq p \leq t'-1 \} \cup \{u_{5l-4}, u_{5i-3} : 1 \leq l \leq t' \}.$\\
In all these possibilities, it is easy to verify that $D_{2,p}$ is a
$2$-paired dominating set. Further, the cardinality of $D_{2,p}$ in
each case is $2\lceil \frac{5nm-9n}{30}\rceil$. Hence,
\begin{equation}
\begin{split}
\gamma_{p}(f_{n\times m}) \leq 2\lceil \frac{5nm-9n}{30}\rceil.
\end{split}
\end{equation}\\
Now we prove the lower bound of $2$-distance paired dominating set.\\
Let $D_{2, p}=\{x_{i}, y_{i} :1 \leq i\leq q\}$ be a $2$-distance
paired dominating set. According to Lemma \ref{lb of Dp}, $D_{2,p}$
contain at least $\lceil \frac{m-6}{6}\rceil$ vertices from each
$V_{i}$ of $C_{i, m}$. If $m\equiv\ 3,$ $(\textrm{mod}\ 6)$, then
$\lceil \frac{m-6}{6}\rceil$ pair of  vertices dominate $m-3$
vertices from each $V_{i}$ of $C_{i, m}$, $where 1 \leq i\leq n.$
The number of vertices which are yet to be dominated in each $C_{i}$
are $3$, from which $2$ vertices are of degree $4$ and one vertex of
degree $2$. Suppose that these vertices are $v_{i, 1}, u_{i}$ and
$u_{i+1}$. To dominate these vertices we choose pair of vertices of
type $(4, 4)$ from $5$ consecutive copies of $C_{i}$ because each
pair of type $(4, 4)$ dominates $16$ vertices of $f_{n \times m}$
Since $\langle D_{2,p}\rangle$ is a perfect matching, so it contains
only the non adjacent edges. This implies that
\begin{equation*}
\begin{split}
q & \geq n\lceil \frac{m-6}{6}\rceil + \lceil \frac{n}{5}\rceil\\
& = n(\frac{m-3}{6}) + \lceil\frac{n}{5}\rceil\\
&=\frac{nm-3n}{6}+\lceil\frac{n}{5}\rceil\\
 &= \lceil \frac{5nm-9n}{30}\rceil
\end{split}
\end{equation*}
which implies that $q  \geq \lceil \frac{5nm-9n}{30}\rceil$. Therefore
\begin{equation}
\begin{split}
\gamma_{p}^{2}(f_{n \times m}) \geq 2\lceil \frac{5nm-9n}{30}\rceil.
\end{split}
\end{equation}\\
From Equation $15$ and $16$, it is clear that $$\gamma_{p}^{2}(f_{n \times m})=
2\lceil \frac{5nm-9n}{30}\rceil.$$\\
\textbf{Case 5}: $m\equiv 4 (\textrm{mod}\ 6)$.\\
Let $t^{'}= \lceil\frac{n}{4}\rceil.$ If $n =4t\,\ \forall$ $t \geq 1$, then define\\
$D_{2, p} =\{v_{4i-3, 6j-1}, v_{4i-3, 6j}, v_{4i-2, 6j-1}, v_{4i-2, 6j}, v_{4i-1, 6j-2}, v_{4i-1,
6j-1}, v_{4i, 6j-2}, v_{4i, 6j-1} : 1 \leq i \leq t',\,\ 1\leq j\leq t\}  \cup\{u_{4l-3}, u_{4l-2} :1 \leq l \leq t'\}.$\\
 If $n=5$, define\\
$D_{2, p}=\{v_{3i-2, 6j-1}, v_{3i-2, 6j}, v_{3i-1, 6j-1}, v_{3i-1,
6j}, v_{3, 6j-2}, v_{3, 6j-1}: 1 \leq j \leq t :1 \leq i \leq
t'\} \cup\{u_{1}, u_{2}, u_{4}, u_{5} \}.$\\
If $n =4t+1\,\ \forall$ $t \geq 2$, define\\
$D_{2, p} =\{v_{4i-3, 6j-1}, v_{4i-3, 6j}, v_{4i-2, 6j-1}, v_{4i-2,
6j}, v_{n-1, 6j-1}, v_{n-1, 6j}, v_{n, 6j-1}, v_{n, 6j} : 1 \leq i
\leq t'-1,\,\ 1\leq j\leq t\} \cup\{v_{4i-1, 6j-2}, v_{4i-1, 6j-1},
v_{4p, 6j-2}, v_{4p, 6j-1} : 1 \leq i \leq t'-1,\,\ 1\leq j\leq
t,\,\ 1\leq p\leq t'-2\}
\cup\{u_{4l-3}, u_{4l-2} :1 \leq l \leq t'-1\} \cup\{u_{n-1}, u_{n}\}.$\\
If $n =4t+2\,\ \forall$ $t \geq 1$, define\\
$D_{2, p} =\{v_{4i-3, 6j-1}, v_{4i-3, 6j}, v_{4i-2, 6j-1}, v_{4i-2,
6j} : 1 \leq i \leq t',\,\ 1\leq j\leq t\} \cup\{v_{4i-1, 6j-2},
v_{4i-1, 6j-1},
v_{4i, 6j-2}, v_{4i, 6j-1} : 1 \leq i \leq t'-1,\,\,\,\ 1\leq j\leq t\} \cup\{u_{4l-3}, u_{4l-2} :1 \leq l \leq t'\}.$\\
If $n=3$, then\\
$D_{2,p}=\{v_{1, 6j-1}, v_{1, 6j}, v_{2, 6j-1}, v_{2, 6j}, v_{3,
6j-2}, v_{3, 6j-1} : 1 \leq i \leq t',\,\ 1\leq j\leq t\}
\cup\{u_{4l-3}, u_{4l-2} :1 \leq l \leq t'\}.$\\
If $n =4t+3\,\ \forall$ $t \geq 1$, then\\
$D_{2, p} =\{v_{4i-3, 6j-1}, v_{4i-3, 6j}, v_{4i-2, 6j-1}, v_{4i-2,
6j} : 1 \leq i \leq t',\,\ 1\leq j\leq t\} \cup\{v_{4i-1, 6j-2},
v_{4i-1, 6j-1},
v_{4p, 6j-2}, v_{4p, 6j-1} : 1 \leq i \leq t',\,\ 1\leq j\leq t,\,\ 1\leq p\leq t'-1\} \cup\{u_{4l-3}, u_{4l-2} :1 \leq l \leq t'\}.$\\
In all these possibilities, it is easy to verify that $D_{2,p}$ is a
$2$-paired dominating set. Further, the cardinality of $D_{2,p}$ in
each case is $2\lceil \frac{2nm-5n}{12}\rceil$, Hence
\begin{equation}
\begin{split}
\gamma_{p}^{2}(f_{n \times m}) \leq 2\lceil \frac{2nm-5n}{12}\rceil.
\end{split}
\end{equation}\\
Now we prove the lower bound of $2$-distance paired dominating set.\\
Let $D_{2,p}=\{x_{i}, y_{i} :1 \leq i\leq q\}$ be a paired
dominating set. According to Lemma \ref{lb of Dp}, $D_{2,p}$ contain
at least $\lceil \frac{m-6}{6}\rceil$ vertices from each $V_{i}$ of
$C_{i, m}$. If $m\equiv4,$ $(\textrm{mod}\ 6)$, then $\lceil
\frac{m-6}{6}\rceil$ pair of  vertices dominate $m-4$ vertices from
each $V_{i}$ of $C_{i, m}$, $where 1 \leq i\leq n.$ The number of
vertices which are yet to be dominated in each $C_{i}$ are $4$, from
which $2$ vertices are of degree $4$ and other vertices of degree
$2$. Suppose that these vertices are $v_{i, 1}, v_{i, 2}, u_{i}$ and
$u_{i+1}$. To dominate these vertices  either $v_{i, 1}, v_{i, 2}$
belongs to $D_{2,p}$ or $u_{i}, u_{i+1}$ belongs to $D_{2,p}$. If
$v_{i, 1}, v_{i, 2}$ belongs to $D_{2,p}$, then the only vertices
which are dominated by these vertices are $u_{i}$ and $u_{i+1}$. If
$u_{i}, u_{i+1}$ belongs to $D_{2,p}$ then these vertices dominate
$13$ vertices of $4$ consecutive copies of $C_{i}$. Thus  we choose
pair of vertices of type $(4, 4)$ from $4$ consecutive copies of
$C_{i}$. Since $\langle D_{2,p}\rangle$ is a paired dominating set
and has perfect matching, so it contains only the non adjacent
edges. This implies that
\begin{equation*}
\begin{split}
|D_{2,p}|& \geq n\lceil \frac{m-6}{6}\rceil + \lceil \frac{n}{4}\rceil\\
& = n(\frac{m-4}{6}) + \lceil\frac{n}{4}\rceil\\
&=\frac{nm-4n}{6}+\lceil\frac{n}{5}\rceil\\
 &= \lceil \frac{2nm-5n}{12}\rceil
\end{split}
\end{equation*}
which implies that $q  \geq \lceil \frac{2nm-5n}{12}\rceil$. Therefore
\begin{equation}
\begin{split}
\gamma_{p}^{2}(f_{n \times m}) \geq 2\lceil \frac{2nm-5n}{12}\rceil.
\end{split}
\end{equation}\\
From Equation $17$ and $18$, it is clear that $$\gamma_{p}^{2}(f_{n \times m})=
2\lceil \frac{2nm-5n}{12}\rceil.$$\\
\textbf{Case 6}: $m\equiv 4 (\textrm{mod}\ 6)$.\\
Let $t^{'}= \lceil\frac{n}{3}\rceil.$\\
If $n=3$, define\\
$D_{2,p}=\{v_{1, 6j-1}, v_{1, 6j}, v_{2, 6j-1}, v_{2, 6j}, v_{3, 6j-2}, v_{3, 6j-1} : 1\leq j\leq t\} \cup\{u_{1}, u_{2}\}.$\\
If $n=4$, define\\
$D_{2, p}=\{v_{i, 6j-1}, v_{i, 6j} :1 \leq i \leq n,\,\ 1\leq j\leq t\} \cup\{u_{1}, u_{2}, u_{3}, u_{4}\}.$\\
If $n=5$, define\\
$D_{2, p}=\{v_{1, 6j-1}, v_{1, 6j}, v_{2, 6j-1}, v_{2, 6j}, v_{4,
6j-1}, v_{4, 6j}, v_{5, 6j-1}, v_{5, 6j}, v_{3, 6j-2}, v_{3, 6j-1} :
1\leq j\leq t\}
 \cup\{u_{1}, u_{2},u_{3}, u_{4}\}.$\\
For $n = 3t\,\ \forall$ $t \geq 2$, define\\
$D_{2, p}=\{v_{3i-2, 6j-1}, v_{3i-2, 6j}, v_{3i-1, 6j-1}, v_{3i-1,
6j}, v_{3i, 6j-2}, v_{3i, 6j-1} : 1 \leq i \leq t',\,\ 1\leq j\leq
t\}
 \cup\{u_{3l-2}, u_{3l-1} : 1 \leq l \leq t'\}.$\\
For $n = 3t+1\,\ \forall$ $t \geq 2$, define\\
$D_{2, p}=\{v_{3i-2, 6j-1}, v_{3i-2, 6j}, v_{3i-1, 6j-1}, v_{3i-1, 6j}, v_{n-1, 6j-1}, v_{n-1, 6j}, v_{n, 6j-1}, v_{n, 6j} : 1 \leq i \leq
t'-1,\,\ 1\leq j\leq t\} \cup\{v_{3i, 6j-2}, v_{3i, 6j-1} : 1 \leq i \leq t'-2,\,\ 1\leq j\leq t\} \cup\{u_{3l-2}, u_{3l-1}, u_{n-1}, u_{n} : 1 \leq l \leq t'-1\}.$\\
For $n = 3t+2\,\ \forall$ $t \geq 2$, define\\
$D_{2, p}=\{v_{3i-2, 6j-1}, v_{3i-2, 6j}, v_{3i-1, 6j-1}, v_{3i-1,
6j} : 1 \leq i \leq t',\,\ 1\leq j\leq t\} \cup\{v_{3i, 6j-2},
v_{3i, 6j-1}
: 1 \leq i \leq t'-1,\,\ 1\leq j\leq t\} \cup\{u_{3l-2}, u_{3l-1} : 1 \leq l \leq t'\}.$\\
In all these possibilities, it is not difficult to see that
$D_{2,p}$ is a $2$-paired dominating set. Further, the cardinality
of $D_{2,p}$ in each case is $2\lceil \frac{nm-3n}{6}\rceil$. Hence
\begin{equation}
\begin{split}
\gamma_{p}^{2}(f_{n \times m}) \leq 2\lceil \frac{nm-3n}{6}\rceil.
\end{split}
\end{equation}\\
Now we give the lower bound of $2$-distance paired dominating set.\\
Let $D_{2,p}=\{x_{i}, y_{i} :1 \leq i\leq q\}$ be a paired
dominating set. According to Lemma \ref{lb of Dp}, $D_{2,p}$ contain
at least $\lceil \frac{m-6}{6}\rceil$ vertices from each $V_{i}$ of
$C_{i, m}$. If $m\equiv4,$ $(\textrm{mod}\ 6)$, then $\lceil
\frac{m-6}{6}\rceil$ pair of  vertices dominate $m-5$ vertices from
each $V_{i}$ of $C_{i, m}$, $where 1 \leq i\leq n.$ The number of
vertices which are yet to be dominated in each $C_{i}$ are $5$, from
which $2$ vertices are of degree $4$ and other vertices of degree
$2$. Suppose that these vertices are $v_{i, 1}, v_{i, 2}, v_{i,
m-1}, u_{i}$ and $u_{i+1}$. To dominate these vertices either $v_{i,
1}, v_{i, m-2}\in D_{2,p}$ or $u_{i}, u_{i+1}\in D_{2,p}$. In both
these cases, each $C_{i, m}$ has at least two vertices from each
$C_{i}$ belong to $D_{2,p}$. Since, each vertex of degree $4$ belong
to neighboring cycle, therefore each vertex of degree $4$ belong to
$D_{2,p}$. Further, $\langle D_{2,p}\rangle$ has perfect matching.
Thus only edges that are not adjacent to each other can belong to
$D_{2,p}$. Thus we choose pair of vertices of type $(4, 4)$ from $3$
consecutive copies of $C_{i}$. Since $\langle D_{2,p}\rangle$ is a
paired dominating set and has perfect matching, so it contains only
the non adjacent edges. This implies that
\begin{equation*}
\begin{split}
q & \geq n\lceil \frac{m-6}{6}\rceil + \lceil \frac{n}{3}\rceil\\
& = n(\frac{m-5}{6}) + \lceil\frac{n}{3}\rceil\\
&=\frac{nm-5n}{6}+\lceil\frac{n}{3}\rceil\\
 &= \lceil \frac{nm-3n}{6}\rceil
\end{split}
\end{equation*}
which implies that $q  \geq \lceil \frac{nm-3n}{6}\rceil$. Therefore\\
\begin{equation}
\begin{split}
\gamma_{p}^{2}(f_{n \times m}) \geq 2\lceil \frac{nm-3n}{6}\rceil.
\end{split}
\end{equation}\\
From Equation $19$ and $20$, it is clear that $$\gamma_{p}^{2}(f_{n \times m})=
2\lceil \frac{nm-5n}{6}\rceil.$$\\

\end{proof}
In Figure $2$, the vertices (dark) of $2$-distance paired dominating set of the graph $f{n\times m}$ are shown.
\begin{figure}[!ht]\label{D1}
\begin{center}
        \centerline
         {\includegraphics[width=12cm]{f2.md}}
        \caption{The 2-distance paired dominating set of $f_{n\times m}$ }
        \end{center}
            \end{figure}

\end{document}